\documentclass[11pt]{amsart}
\usepackage{amssymb,latexsym,amsmath,amsthm,amsfonts}
\usepackage{xcolor}

\newtheorem {theo} {\bfseries Theorem} [section]
\newtheorem {prop} [theo] {\bfseries Proposition}
\newtheorem {coro} [theo] {\bfseries Corollary}
\newtheorem {lem} [theo] {\bfseries Lemma}
\newtheorem {defi} {\bfseries Definition}[section]
\newtheorem{exam} {\bfseries Example}[section]

\newtheorem{rem}{\bfseries Remark}[section]
\numberwithin{equation}{section}
\newcommand{\R}{\mathbb R}
\def\rank{\mathrm{rank}\,}
\newcommand{\ip}[2]{\langle\,#1\, , \, #2 \, \rangle}

\begin{document}
\title{Frame Phase-retrievability and Exact phase-retrievable frames}
\author{Deguang Han}
\address{Department of Mathematics\\
University of Central Florida\\ Orlando, FL 32816}
\email{deguang.han@ucf.edu}
\author{Ted Juste}
\address{Department of Mathematics\\ University of Central Florida\\Orlando, FL 32816}
\email{tjuste@knights.ucf.edu}
\author{Youfa Li}
\address{College of Mathematics and Information Science\\
Guangxi University,  Naning, China }
\email{youfalee@hotmail.com}
\author{Wenchang Sun}
\address{School of Mathematical Sciences,\\
Nankai University,  Tianjin, China }
\email{sunwch@nankai.edu.cn}
\thanks{This work is partially supported by the NSF grant DMS-1403400}

\date{\today}

\begin{abstract} An exact phase-retrievable frame $\{f_{i}\}_{i}^{N}$ for an $n$-dimensional Hilbert space is a phase-retrievable frame that  fails to be phase-retrievable if any
one element is removed from the frame. Such a frame could have different lengths. We shall prove that for the real Hilbert space case,  exact phase-retrievable frame of length $N$ exists for every $2n-1\leq N\leq n(n+1)/2$. For arbitrary frames  we introduce the concept of redundancy with respect to its phase-retrievability and the concept of frames with exact PR-redundancy. We investigate  the phase-retrievability by studying its maximal phase-retrievable subspaces with respect to a given frame which is not necessarily phase-retrievable. These  maximal PR-subspaces could have different dimensions. We are able to identify the one with the largest dimension,  which can be considered as a generalization of the characterization for phase-retrievable frames. In the basis case, we prove that if $M$ is a $k$-dimensional PR-subspace, then $|supp(x)| \geq k$  for every nonzero vector  $x\in M$. Moreover, if $1\leq k< [(n+1)/2]$, then a $k$-dimensional PR-subspace is maximal if and only if there exists a vector $x\in M$ such that $|supp(x) | = k$.
\end{abstract}
\maketitle

\section{Introduction}
\setcounter{figure}{0} \setcounter{equation}{0}

A finite sequence $\mathcal{F} = \{f_i \}^N_{i=1}$  of vectors in an $n$-dimensional Hilbert space $H$  is called a {\it
frame} for $H$ if there are two constants $0< C_{1}\leq C_{2}$ such that
$$C_{1}\|f\|^2 \leq \sum_{i=1}^N |\langle f,f_i  \rangle|^2 \leq C_{2}\|f\|^2 $$
holds for every $f\in H$. Equivalently, a finite sequence is a frame for $H$ if and only if it is a spanning set of $H$.  Two frames $\{f_{i}\}_{i=1}^{N}$ and $\{g_{i}\}_{i=1}^{N}$ are called similar if there exists an invertible operator $T$ such $g_{i} = Tf_{i}$ for every $i$.  For a given frame $\mathcal{F} =\{f_{i}\}_{i=1}^{N}$, the {\it spark} of $\mathcal{F}$ is the cardinality of the smallest linearly dependent subset of the frame. A {\it full-spark frame} is a frame whose spark is $n+1$, i.e., every $n$-vectors in $\mathcal{F}$ are linearly independent.

In recent years, frames have been extensively studied in the context of  the so-called phase-retrieval problem which  arises in various fields of science and engineering applications, such as X-ray crystallography, coherent diffractive imaging, optics and  many more.   The problem asks to recover a signal of interest from the magnitudes of its linear or nonlinear measurements.  For  the linear measurements with a frame $\{f_{i}\}_{i=1}^{N}$,  one wants to reconstruct $f$ from its intensity measurements  $\{|\langle f, f_{i}\rangle|\}_{i=1}^{N}$.  Clearly the intensity measurements are the same for both $f$ and $\lambda f$ for every unimodular scalar $\lambda$. Therefore the phase retrieval problem asks to recover $f$ up to an unimodular scalar. We refer to \cite{Ba1}-\cite{Wang-Xu} and the reference therein for some historic background of the problem and some recent developments on this topic.

\begin{defi} A  frame $\{f_{i}\}_{i=1}^{N}$ for  a Hilbert space $H$ is called {\it phase retrievable} if the induced quotient map $\mathcal{A}: H/\mathbb{T} \rightarrow \mathbb{R}^{N}$ defined by $\mathcal{A}(f/\mathbb{T}) = \{|\langle f, f_{i}\rangle|\}_{i=1}^{N}$ is injective, where $\mathbb{T} = \{\lambda\in \mathbb{R}: |\lambda| = 1\}$.
\end{defi}

There are few basic concepts when talking about frames or frame sequences:  signal recoverability, redundancy and the exactness of frames. The signal recoverability of a sequence $F = \{f_{i}\}$ can be measured by the space spanned by $F$,  and the redundancy of a finite frame $F = \{f_{i}\}_{i=1}^{N}$ for an $n$-dimensional Hilbert space can be measured by $N/n$.  An exact frame for a Hilbert space $H$ is a frame such that it fails to be a frame if we remove any one element from $F$. So exact frames are precisely the bases or the frames with redundancy one.  These concepts naturally lead us to the following questions when dealing with the phase-retrieval problem:  Given a frame $F = \{f_{i}\}_{1}^{N}$ (which may not be phase-retrievable). How to measure its phase-retrievability? How to measure  its redundancy with respect to the phase-retrievability and what  can be said about those phase-retrievable frames that have the exact PR-redundancy?

 Even a frame  is not phase-retrievable, it is still possible that it can be used to perform phase retrieval for some subsets of the Hilbert space. So for the purpose of theory development and practical applications,  it seems natural to investigate the subsets on which phase-retrieval can be performed with respect to a given frame (usually your favorite one but not phase-retrievable). In this paper we initiate the  study on the maximal phase-retrievable subspaces for a given frame. This consideration naturally leads to the concept of frame redundancy with respect to the phase-retrievability and the notion of exact phase-retrievable frames.  Unlike exact frames, exact phase-retrievable frames could have different lengths.  This paper will be focused on the existence problem of exact phase-retrievable frames  (or  more generally,  the frames with the exact PR-redundancy) with all the possible lengths,  and the maximal phase-retrievable subspaces of all possible dimensions.

There are two well-known necessary and sufficient conditions for phase retrievable frames(c.f. \cite{Ba2, Ba3, Ba4, Ba5, Ba6}). The first one is given in terms of the so-called ``complement property": A frame $\{f_{i}\}_{i=1}^{N}$ is said to have the {\it complement property} if for every $\Omega \subseteq \{1, ... , N\}$ we have either $\{f_{i}\}_{i\in\Omega}$ or $\{f_{i}\}_{i\in\Omega^{c}}$ spans $H$.

\begin{prop} \label{thm-1.1} The complement property is necessary for a frame to be phase-retrievable. It is also sufficient for real Hilbert spaces.
\end{prop}

The second condition is based on the rank-one operator lifting of the frame $\{f_{i}\}_{i=1}^{N}$. For each $f, g\in H$, let $f\otimes g$ be the rank-one operator defined by $(f\otimes g) x = \langle x, g\rangle f$ for every $x\in H$.
In what follows we use $\langle A, B\rangle  = tr(AB^{*})$ to denote the Hilbert-Schimidt inner product on the space of $n\times n$ matrices and $\mathcal{S}_{2}$ be the set of all the Hermitian $n\times n$ matrices with rank less than or equal to $2$. Given a sequence $\mathcal{F}=\{f_{i}\}_{i=1}^{N}$ in $H$ and let $\Theta_{L(\mathcal F)}$ be the analysis operator of $L(\mathcal{F}):=\{L(f_{i})\}$, where $L(x): = x\otimes x$.  From the definition of phase-retrievable frames, it is easy to see  get the following:

\begin{prop}
A  frame $\{f_{i}\}_{i=1}^{N}$ is phase-retrievable if and only if $$ker(\Theta_{L(\mathcal F)})\cap \mathcal{S}_{2} = \{0\}.$$
\end{prop}

The above characterization indicates that $ker(\Theta_{L(\mathcal F)})\cap \mathcal{S}_{2}$ seems to be a good candidate to measure the phase-retrievability for a frame $\mathcal{F}$.  This motivated us to introduce the following  concept of redundancy with respect to the phase-retrievability (or PR-redundancy) and the concept of frames with the exact PR-redundancy property. Let  $\mathcal{F}=\{f_{i}\}_{i=1}^{N}$  be a frame for  $H$. For each subset $\Lambda$ of $\{1, ... , N\}$, let $\mathcal{F}_{\Lambda} = \{f_{i}\}_{i\in \Lambda}$ and use $|\Lambda |$ to denote the cardinality of $\Lambda$.

\begin{defi} Given a frame $\mathcal{F}=\{f_{i}\}_{i=1}^{N}$   for  $H$. Let $k$ be the smallest integer such that there exists a subset $\Lambda$ of $\{1, ... , N\}$ with the property  that $|\Lambda| = k$ and
$$ker(\Theta_{L(\mathcal{F}_{\Lambda})})\cap \mathcal{S}_{2}  = ker(\Theta_{L(\mathcal F)})\cap \mathcal{S}_{2} .$$ Then we call $N/k$ the {\it PR-redundancy} of $\mathcal{F}$. A frame $\mathcal{F}=\{f_{i}\}_{i=1}^{N}$   for  $H$
is said to have the {\it exact PR-redundancy property} if its PR-redundancy is $1$.
A phase-retrievable frame  with the exact PR-redundancy will be called an {\it exact  phase-retrievable frame}.
\end{defi}

Given a frame  $\mathcal{F} =\{f_{i}\}_{i=1}^{N}$ for $H$. From the above definition we have the following: (i) There exists a subset $\Lambda$ of $\{1, ... , N\}$ such that $\mathcal{F}_{\Lambda}$ is a frame for $H$ with the exact PR-redundancy property.  (ii) $\mathcal{F}$ has the exact PR-redundancy property if and only if for any proper subset $\Lambda$ of $\{1, ..., N\}$, there exist two vectors $x, y\in H$ such that
$|\langle x,  f_{j}\rangle | = |\langle y,  f_{j}\rangle | $ for every $j\in\Lambda$, but $|\langle x,  f_{i}\rangle | \neq |\langle y,  f_{i}\rangle |$ for some $i\in\Lambda^{c}$. (iii) If $\mathcal{F}$  is phase-retrievable, then it is an exact phase-retrievable frame  if and only if $\mathcal{F}_{\Lambda}$ is no longer phase-retrievable for any proper subset $\Lambda$ of $\{1, ... , N\}$.

In what follows we always assume that $H = \Bbb{R}^{n}$ and use $\mathbb{H}_{n}$ to denote the space of all the $n\times n$ Hermitian matrices.

\begin{lem} \label{lem-00}  If a frame $\mathcal{F} = \{f_{i}\}_{-1}^{N}$ for $\R^{n}$ has the exact PR-redundancy property, then $\{L(f_{i})\}_{i=1}^{N}$ is a linearly independent set (and hence $N \leq \dim \mathbb{H}_{n} = n(n+1)/2$).  The converse is false.
\end{lem}

\begin{proof} If $\{L(f_{i})\}_{i=1}^{N}$ is linearly dependent, then there exists a proper subset $\Lambda$ of $\{1, ... , N\}$ such that $\mathrm{span}\,  \{L(f_{i}): i\in \Lambda\} = \mathrm{span}\,  \{L(f_{i}): 1\leq i\leq N\}$. This implies that $ker(\Theta_{L(\mathcal{F}_{\Lambda})}) = ker(\Theta_{L(\mathcal F)})$. Hence $\mathcal{F}$ does not have the exact PR-redundancy property. Therefore $\{L(f_{i})\}_{i=1}^{N}$ is a linearly independent  set.

Let $n\geq 3$. Then $2n-1 < n(n+1)/2$. Let $\{f_{1}, ... , f_{2n-1}\}$ be a phase-retrievable frame for $H$ which clearly must have the exact PR-redundancy property. Thus  $\{L(f_{i})\}_{i=1}^{2n-1}$ is linearly independent. Since $\dim \mathbb{H}_{n} = n(n+1)/2$ and $\mathrm{span}\,  \{L(x): x\in H\} = \mathbb{H}_{n}$, we can extend  $\{L(f_{i})\}_{i=1}^{2n-1}$  to a basis  $\{L(f_{i})\}_{i=1}^{n(n+1)/2}$.
But clearly $\mathcal{F} = \{f_{i}\}_{i=1}^{n(n+1)/2}$ does not have the exact PR-redundancy.
\end{proof}

Lemma \ref{lem-00} immediately implies the following length bound for exact phase-retrievable frames.

\begin{coro} If $\mathcal{F}=\{f_{i}\}_{i=1}^{N}$ is an exact phase-retrievable frame for $\R^{n}$, then $2n-1\leq N \leq n(n+1)/2$.
\end{coro}

This leads to the question about the attainable lengths for exact phase-retrievable frames. Our first main result shows that  every $N$ between $2n-1$ and $N \leq n(n+1)/2$ is attainable, i.e., there exists an exact phase-retrievable frame of length $N$ for every such $N$.

 It is known that for each $N\geq n$, the set of full-spark frames of length $N$ (i.e., every $n$ vectors in $\mathcal{F}$ are linearly independent) for an open dense subset in the direct sum space $H^{(N)} : = H \oplus ... \oplus H$, $N$-copies). It is clear that if $N > 2n-1$ and $\mathcal{F} = \{f_{i}\}_{i=1}^{N}$  has the full spark, then $N$ can not be an exact phase-retrievable frame. Therefore the  set of  exact phase-retrievable frames of length $N$ has measure zero, and so the existence proof of exact phase-retrievable frames is quite subtle, as demonstrated in section 2.

For a non-phase-retrievable frame $\mathcal{F}$, researchers have been interested in identifying the subsets of the signal space such that phase-retrieval can be performed by the frame on these subsets. A typical example is the subset of sparse signals (e.g. \cite{Eldar, Wang-Xu}). In order to have a better understanding about the phase-retrievability,    here we are interested in the problem of identifying the largest subspaces  $M$ such that $\mathcal{F}$ does the phase-retrieval for all the signals in $M$. For this purpose we introduce the following definition:

\begin{defi}  Let $\mathcal{F}=\{f_{i}\}_{i=1}^{N}$  be a frame for  $H$ and $M$ is a subspace of $H$. We say that $M$ is a {\it phase-retrievable subspace} with respect to $\mathcal{F}$ if $\{P_{M}f_{i}\}_{i=1}^{N}$ is a phase-retrievable frame for $M$, where $P_{M}$ is the  orthogonal projection  from $H$ onto $M$. A phase-retrievable subspace $M$ is called maximal if it is not a proper subspace of any other phase-retrievable subspaces with respect to $\mathcal{F}$.
\end{defi}

We will use the abbreviation ``$\mathcal{F}$-PR subspace " to denote a phase-retrievable subspace with respect to $\mathcal{F}$. Given a frame $\mathcal{F}$. Naturally we would like to know the answers to the following questions: What are possible dimensions $k$ such that there exists a $k$-dimensional  maximal $\mathcal{F}$-PR subspace? What is the largest (or the smallest)  dimension for all the maximal  $\mathcal{F}$-PR subspaces?

As an motivating example, we will show that if  $\mathcal{F} =\{f_{i}\}_{i=1}^{n}$ is a basis for $H$, then there exists a $k$-dimensional  maximal $\mathcal{F}$-PR subspace if and only if $1\leq k \leq [(n+1)/2]$, where $[a]$ denotes the integer part of $a$ . For any general frame $\mathcal{F}$, we will identify the largest $k$ such that there exists a $k$-dimensional  maximal $\mathcal{F}$-PR subspace. This leads to a generalization of Proposition \ref{thm-1.1}. In the case that $\mathcal{F} =\{f_{i}\}_{i=1}^{n}$ is an orthonormal  basis, we show that if $M$ is a $\mathcal{F}$-PR subspace, then the support $supp(x)$  (with respect to the dual basis) of every nonzero vector $x$ in $M$ has the cardinality greater than or equal to $k$. Moreover, we will prove that for any given vector $x$ with $|supp(x)| = k$, there exists a $k$-dimensional maximal $\mathcal{F}$-PR subspace $M$ containing $x$.  This support condition is also necessary in the case that  $k < [(n+1)/2]$, i.e, in this case we have that a $k$-dimensional $\mathcal{F}$-PR subspace $M$ is maximal if and only if there exists an nonzero vector $x$ in $M$ whose support has the cardinality  $k$.

The following simple property will be needed in the rest of the paper.

\begin{lem} \label{lem-02} Suppose that $H$ is the direct sum of two subspaces $X$ and $Y$. If $\mathcal{F}_{1}$ is a frame
for $X$ with the exact PR-redundancy property  and $\mathcal{F}_{2}$ is a frame for $Y$ with the exact PR-redundancy property,
then $\mathcal{F} = \mathcal{F}_{1}\cup \mathcal{F}_{2}$ is a frame for $H$ with the exact PR-redundancy property.
\end{lem}
\begin{proof} By passing to a similar frame we can assume that $Y = X^{\perp}$.  Clearly $\mathcal{F}$ is a frame
for $H$. Now assume that  a vector $f$ is removed from $\mathcal{F}_{1}$. Since $\mathcal{F}_{1}$ is a frame for $X$ with
the exact PR-redundancy property,  there exists some nonzero operator $A = u\otimes u - v\otimes v$ with $u, v\in X$
such that $A\in ker (\Theta_{L(\mathcal{F}_{1}\setminus \{f\})})$ and $A\notin ker (\Theta_{L(\mathcal{F}_{1})})$.
Since $Y\perp X$, we also have $A\in ker (\Theta_{L(\mathcal{F}_{2})})$. This implies that
$A\in ker (\Theta_{L(\mathcal{F}\setminus \{f\})})$ and $A\notin ker (\Theta_{L(\mathcal{F})})$.
The same argument works if we remove one element from $\mathcal{F}_{2}$. Thus $\mathcal{F}$ has the exact PR-redundancy property.

\end{proof}

\section{Exact Phase-retrievable Frames}

In this section we prove the existence theorem for exact phase-retrievable frames of length $N$ with $2n-1\leq N \leq n(n+1)/2$.

\begin{theo} \label{main-thm-1} For every integer $N$ with $2n-1\leq N \leq n(n+1)/2$, there exists an exact phase-retrievable frame of length $N$.
\end{theo}

Before giving a proof for the above theorem, we introduce some preliminary results.
We use the following notations for matrices: $A(I,J)$ is the submatrix of $A$ consisting of the entries with row indices in $I$ and column indices in $J$.
$A(:,J) = A(\{1,\ldots,n\}, J)$ and $A(i,j) = A(\{i\}, \{j\})$

\begin{lem}\label{Lm:L1}
Let $f(x_1,\ldots,x_n)$ be a polynomial
and  $a_i$ be independent continuous random variables.
Then $f(a_1,\ldots,a_n) \ne 0$ almost surely.
\end{lem}

\begin{proof}
The conclusion can be proved by induction on $n$ and we omit the details.
\end{proof}

\begin{lem}\label{Lm:L2}
Let $A$ be an $n\times m$ random matrix such that $\rank(A) = r$ almost surely.
Let $B$ be an $(n+1)\times (m+1)$ matrix such that $B(1..n,1..m) = A$
and $B(n+1,m+1)$ is a continuous random variable which is independent of the entries of $A$.
Then we have $\rank(B) \ge r+1$ almost surely.
\end{lem}

\begin{proof}
Let $\Omega$   be the sample space.
Since $A$ has only finitely many submatrices and $\rank(A) = r$ almost surely, there is a partition $\{\Omega_i\}_{i=1}^N$ of $\Omega$
such that for each $1\le i\le N$, there is an $r\times r$ submatrix $A_i$ which is of rank $r$ almost surely on $\Omega_i$.
Therefore, the submatrix of $A$ consisting of rows and columns in $A_i$ and the $(n+1)$-th row and the $(m+1)$-th column
is of rank $r+1$ almost surely on $\Omega_i$, thanks to Lemma~\ref{Lm:L1}. This completes the proof.
\end{proof}

The following lemma can be proved similarly, which we leave to interested readers.
\begin{lem}\label{Lm:L3}
Let $A$ be an $n\times m$ random matrix such that $\rank(A) = r \le n-1$ almost surely.
Let $a$ be an $n$-dimensional vector with entries consisting of continuous independent random variables, which are also independent of the entries of $A$.
Then we have $\rank((A \ a)) = r+1$ almost surely.
\end{lem}

We are ready to give a proof of Theorem~\ref{main-thm-1}.

\begin{proof}[Proof of Theorem~\ref{main-thm-1}] Since every full-spark frame of length $2n-1$ is an exact PR-frame, we only need to prove the theorem for $2n\leq N\leq n(n+1)/2$.
First, we show that for $2n\le N\le n(n+1)/2$, there exist $n\times N$ matrices $A$ such that
\begin{enumerate}
\item [(P1)] $A$ contains the $n\times n$ identity matrix as a submatrix;
\item [(P2)] the rest $N-n$ columns of $A$ consisting of independent continuous random variables or zeros
             and each column contains at least  one $0$ and two non-zero entries;
\item [(P3)] there are exactly $n$ non-zero entries  in every row of $A$;
\item [(P4)] for each $1\le i\le n$, there exist mutually different indices $j_1$, $\ldots$, $j_n$  such that
             $a_{i,j_l}, a_{l,j_l} \ne 0$;
\item [(P5)] columns of $A$ form an exact PR frame with probability $1$.
\end{enumerate}

It is obvious that a phase-retrievable frame which satisfies (P3) is exact.
Let us explain (P4) in more details.

Fix some $i$, say, $i=1$. By (P3), there  exist mutually different indices $j_1$, $\ldots$, $j_n$  such that $a_{1,j_l}\ne 0$
for $1\le l\le n$. (P4) says that every row contains a non-zero entry in such columns and different rows correspond to different columns.

Consider the following example,
\begin{equation}\label{eq:e1}
  A = \begin{pmatrix}
     1  &0  & 0 & a_{1,4}  & a_{1,5}  & 0 \\
     0  &1  & 0 & a_{2,4}  & 0        & a_{2,3}\\
     0  &0  & 1 & 0        & a_{3,5}  & a_{3,3}
  \end{pmatrix},
\end{equation}
where $a_{i,j}$ are independent continuous random variables.
For $i=1$, set $\{j_1,j_2,j_3\} = \{1,4,5\}$. Then we have $a_{1,j_l}, a_{l,j_l} \ne 0$ for $1\le l\le 3$.

It is easy to see that $A$ satisfies (P1)$\sim$ (P5). In other words, such matrix exists for $n=3$.

Now we assume that such matrix $A$ exists for some $n$ and $N$ with $n\ge 3$.
Let us consider the case of $n+1$. We prove the conclusion in the following four steps.


(I). \textcolor{blue}{There is an $(n+1)\times (N+n+1)$ matrix satisfying (P1) $\sim$ (P5)}.

Define the $(n+1)\times (N+n)$ matrix $B$ as follows,
\[
  B = \begin{pmatrix}
         &    a_{1,N+1}& 0          &             &  0        & 0       \\
         &    0        & a_{2,N+2}  &             &  0       &  0            \\
    A    &    0        & 0          &   \cdots    &  0       & 0       \\
         &             & \ldots     &             &  &         \\
         &   0         &0           &             & a_{n,N+n}   &  0            \\
0\ldots 0& a_{n+1,N+1} &a_{n+1,N+2} &             & a_{n+1,N+n} & 1   \\
  \end{pmatrix}.
\]
where all the symbols $a_{i,j}$ are independent continuous random variables.
It is easy to see that $B$ meets (P1) $\sim$ (P4).
It remains to prove that (P5) holds for $B$.

Take some $J\subset \{1,\ldots, N+n+1\}$.
Set
\begin{eqnarray*}
 J^c &=& \{1\le j\le N+n+1:\, j\not\in J\},\\
 J|_N &=& \{j\in J:\, j\le N\}, \\
 J^c|_N &=& \{j\in J^c:\, j\le N\},
\end{eqnarray*}
Without loss of generality, we assume that
$N+n+1\in J^c$.

Suppose that $\rank(B(:, J^c))<n+1$ on some sample set  $\Omega'$  which is of positive probability.
Since $N+n+1\in J^c$,  we have
$\rank(A(:, J^c|_N))<n$ a.s. on $\Omega'$.
Consequently,
$\rank(A(:, J|_N))=n$ a.s. on $\Omega'$.

On the other hand,
Since $N+n+1\in J^c$, not all of $N+1$, $\ldots$, $N+n$ are contained in $J^c$.
 Otherwise, $\rank(B(:,J^c)) = n+1$  a.s. on $\Omega'$.
Hence there is some $1\le i\le n$ such that $N+i\in J$.
By Lemma~\ref{Lm:L2},
$\rank(B(:, J))=n+1$  a.s. on $\Omega'$.

(II). \textcolor{blue}{There is an $(n+1)\times (N+n)$ matrix satisfying (P1) $\sim$ (P5)}.

Since $A$ satisfies (P2), by rearranging columns of $A$, we may assume that $A(:,N) = (0, a_{2,N}, \ldots )^t$, where at least two entries are non-zero.
Define the $(n+1)\times (N+n)$ matrix $B$ as follows,
\[
  B = \begin{pmatrix}
         &    0        &    a_{1,N+1}& 0          &             &  0        & 0       \\
         &    a_{2,N}  &    a_{2,N+1} &0          &             &  0       &  0            \\
   \ldots&    *        &    0        & a_{3,N+2}  &   \cdots    &  0       & 0       \\
         &             &             & \ldots     &             &  &         \\
         &     *       &   0         &0           &             & a_{n,N+n-1}   &  0            \\
         & a_{n+1,N}   & a_{n+1,N+1} &a_{n+1,N+2} &             & a_{n+1,N+n-1} & 1   \\
  \end{pmatrix}.
\]
Again, we only need
 to prove that (P5) holds for $B$.

As in Step I,
we take some $J\subset \{1,\ldots, N+n\}$.
We suppose that
$N+n \in J^c$ and that $\rank(B(:$, $J^c))<n+1$  on some sample set  $\Omega'$  which is of positive probability.
Then we have $\rank(A(:$, $J|_N))=n$ a.s. on $\Omega'$.

If there is some $1\le i\le n$ such that $N+i\in J$, then we have
$\rank(B(:$, $J))=n+1$  a.s. on $\Omega'$, thanks to Lemma~\ref{Lm:L2}.

Next we assume that $N+i\in J^c$ for $1\le i\le n$.
Since $\rank(B(:$, $J^c))<n+1$  a.s. on $\Omega'$, for any $j\le N$ with $A(1,j)\ne 0$, we have $j\in J$,
thanks to Lemma~\ref{Lm:L1}. Similarly we get that $N\in J$.

By setting $i=1$ in (P4), we get mutually different $1\le j_1,\ldots,j_n \le N$ such that $A(1,j_l), A(l,j_l)\ne 0$.
Hence $j_1, \ldots, j_n\in J|_N$. Moreover, $\rank(A(:$, $\{j_1,\ldots,j_n\})) = n$  a.s. on $\Omega'$, thanks to Lemma~\ref{Lm:L1}.
Note that $N\in J|_N$ and $N\ne j_l$ for $1\le l\le n$. By Lemma~\ref{Lm:L2}, we have
\[
  \rank(B(:,\{j_1,\ldots,j_n,N\})) = n+1,\qquad  \mathrm{a.s.\, on\, } \Omega'.
\]
Hence
\[
  \rank(B(:,J)) \ge \rank(B(:,\{j_1,\ldots,j_n,N\})) = n+1,\qquad  \mathrm{a.s.\, on\, } \Omega'.
\]

(III). \textcolor{blue}{There is an $(n+1)\times (N+2)$ matrix satisfying (P1) $\sim$ (P5)}.

By rearranging columns of $A$, we may assume that
\def\labelenumi{(\arabic{enumi})}
\begin{enumerate}
\item $A(:,\{1,\ldots,n\})$ is the $n\times n$ identity matrix (P1),
\item $A(n,N)=0$ and there are at least two non-zero entries in the $N$-th column (P2),
\item $A(i,N-i), A(n,N-i)\ne 0$ for $1\le i\le n-1$ (P4).
\end{enumerate}

Define the $(n+1)\times (N+2)$ matrix $B$ as follows,
\[
  B = \begin{pmatrix}
               &         &    *           &         &     *          &   a_{1,N-1} &     *        & a_{1,N+1} &     0       \\
               &         &    *           &         &     a_{2,N-2} &     *        &     *        & a_{2,N+1} &     0            \\
I_{n\times n}  &         &    *           &         &     *         &     *        &     *        & a_{3,N+1} &     0       \\
               &    \ldots     &          &  \ldots &               &              &              & \ldots    &                      \\
               &         & a_{n-1,N-n+1}  &         &    *          &    *         &     *        &           &     0            \\
               &         & a_{n,N-n+1}    &         &    a_{n,N-2}  &  a_{n,N-1}   &     0        &a_{n,N+1}  &     0            \\
 0\ldots 0     &         & a_{n+1,N-n+1}  &         &  a_{n+1,N-2}  &  a_{n+1,N-1} &  a_{n+1,N}   &0          &  1 \\
  \end{pmatrix}.
\]
Again, we only need
 to prove that (P5) holds for $B$.

As in Step I,
take some $J\subset \{1,\ldots, N+2\}$ and suppose that
$N+2 \in J^c$ and $\rank(B(:,J^c))<n+1$  on some sample set  $\Omega'$  which is of positive probability.
Then we have $\rank(A(:,J|_N))=n$  a.s. on $\Omega'$.

There are three cases.

(i).\,\, $N+1\in J^c$

In this case, we conclude that
\begin{enumerate}
\item [(a)] $\rank(B(1..n,J^c|_N)) \le n-2$, a.s. on $\Omega'$;
\item [(b)] there is some $1\le j_0\le n-1$ such that $N-j_0\in J$.
\end{enumerate}

In fact, if there is some $\Omega''\subset \Omega'$ with positive probability such that
$\rank(B(1..n,J^c|_N))$ $= n-1$ a.s. on $\Omega''$, then we see from Lemma~\ref{Lm:L3}
that  $\rank(B(1..n,J^c|_N\cup\{N+1\})) = n$ a.s. on $\Omega''$.
By Lemma~\ref{Lm:L2}, we get $\rank(B(:,J^c)) = n+1$ a.s. on $\Omega''$, which contradicts with the assumption.
This proves (a).

On the other hand, if $N-j\in J^c$ for any $1\le j \le n-1$,
then the expansion of the determinant of $B(:,\{N-n+1,N-n+2,\ldots, N-1, N+1, N+2\})$
contains the term
$A(n,N+1) \cdot 1 \cdot \prod_{i=1}^{n-1} A(i, N-i)$, which is not zero a.s.
By Lemma~\ref{Lm:L1}, $\rank(B(:,J^c)) = n+1$ a.s. on $\Omega'$. Again, we get a contradiction  with the assumption.
Hence (b) holds.

We see from (a) and (b) that $\rank(B(1..n,J^c|_N\cup\{N-j_0\})) \le n-1$,  a.s. on $\Omega'$.
Since $A$ is a PR frame a.s., we have
$\rank(B(1..n,J|_N\setminus\{N-j_0\})) = n$ a.s. Now we see from Lemma~\ref{Lm:L2}  that
$\rank(B(:,J|_N)) = n+1$   a.s. on $\Omega'$.

(ii).\,\, $N+1 \in J$ and $N-j_0\in J$ for some $0\le j_0 \le n-1$.

Since $\rank(A(:,J|_N))=n$  a.s. on $\Omega'$,
by Lemma~\ref{Lm:L2},
\[
  \rank(B(\{1,\ldots,n\},J|_N\cup\{N+1\}\setminus\{N-j_0\})) = n,\qquad \mathrm{a.s.\, on\, } \Omega'.
\]
Using Lemma~\ref{Lm:L2} again, we get
\[
  \rank(B(:,J|_N\cup\{N+1\})) = n+1, \qquad \mathrm{a.s.\, on\, } \Omega'.
\]
Hence
\[
  \rank(B(:,J)) = n+1,\qquad \mathrm{a.s.\, on\, } \Omega'.
\]

(iii).\,\, $N+1 \in J$ and $N-j\in J^c$ for any $0\le j \le n-1$.

By (P2), there is some $1\le i_0\le n-1$ such that $A(i_0, N)\ne 0$.
Hence the expansion of the determinant of $B(:,\{N-n+1,N-n+2,\ldots, N,N+2\})$
contains the term
$B(n+1,N+2)A(n,N-i_0)A(i_0,N) \prod_{1\le i\le n-1, i\ne i_0} A(i, N-i)$, which is not zero a.s.
By Lemma~\ref{Lm:L1}, $\rank(B(:,J^c)) = n+1$ a.s. on $\Omega'$, which contradicts with the assumption.

(IV).
 \textcolor{blue}{For $2n\le N\le n(n+1)/2$, there exist $n\times N$ matrices  satisfying (P1) $\sim$ (P5)}.

Let $K_n$ be the set of all integers $k$ such that there exists an $n\times k$ matrix $A$ satisfying (P1) $\sim$ (P5).

Since $K_3 \supset \{6\}$, we see from the previous arguments that
\begin{eqnarray*}
  K_4 &\supset& \{8,9,10\}, \\
  K_5 &\supset& \{10,11,12,13,14,15\}.
\end{eqnarray*}
Hence  for $3\le n\le 5$,
\begin{equation}\label{eq:Jn}
K_n \supset \{k:\, 2n\le k\le n(n+1)/2\}.
\end{equation}

Now suppose that (\ref{eq:Jn}) is true for some $n\ge 5$.
Since $2n + (n+1) \le  n(n+1)/2 +2$ for $n\ge 5$, we have
\begin{eqnarray*}
 && \{k+2:\, 2n\le k\le n(n+1)/2\} \cup \{k+n:\, 2n\le k\le n(n+1)/2\} \\
   &&\qquad \cup \{k+n+1:\, 2n\le k\le n(n+1)/2\}\\
  &=&\{k:\, 2(n+1)\le k\le (n+1)(n+2)/2\}.
\end{eqnarray*}
Hence $K_{n+1}\supset  \{k:\, 2(n+1)\le k\le (n+1)(n+2)/2\}$. By induction,  (\ref{eq:Jn}) is true for $n\ge 3$.

Finally, since columns of a randomly generated $n\times (2n-1)$ matrix form an exact PR frame almost surely, we get
the conclusion as desired.
\end{proof}

The following are some explicit examples for $n=5$ and $10\le N\le 15$.
In each case, column vectors of $A$ form an exact PR frame.
Moreover, such matrices correspond to exact PR frames almost surely if the non-zero entries are replaced with
independent continuous random variables.

$(n,N) = (5,10)$:
\[
A=\left(
\begin{array}{rrrrrrr}
     1   &     0    &    0   &     0    &    0    \\
     0   &     1    &    0   &     0    &    0    \\
     0   &     0    &    1   &     0    &    0    \\
     0   &     0    &    0   &     1    &    0    \\
     0   &     0    &    0   &     0    &    1
\end{array}
\begin{array}{rrrrrrr}
    6    &    4     &   2     &  11     &   0    \\
   13    &   10     &   8     &   0     &   3    \\
    7    &    7     &   0     &   9     &   8    \\
   16    &    0     &   8     &  30     &  13    \\
    0    &    4     &  12     &  14     &  18
\end{array}\right).
\]

$(n,N) = (5,11)$:
\[
A=\left(
\begin{array}{rrrrrrr}
     1   &     0    &    0   &     0    &    0    \\
     0   &     1    &    0   &     0    &    0    \\
     0   &     0    &    1   &     0    &    0    \\
     0   &     0    &    0   &     1    &    0    \\
     0   &     0    &    0   &     0    &    1
\end{array}
\begin{array}{rrrrrrr}
    5    &    0  &      3   &    35  &      7    &    0   \\
   18    &    0  &     14   &    27  &      0    &    2   \\
    0    &   23  &      5   &     0  &      1    &   14   \\
    0    &    8  &      0   &    14  &      7    &   14   \\
    0    &    0  &      3   &    30  &      3    &   14
\end{array}\right).
\]

$(n,N) = (5,12)$:
\[
A=\left(
\begin{array}{rrrrrrr}
     1   &     0    &    0   &     0    &    0  &      7   \\
     0   &     1    &    0   &     0    &    0  &      4   \\
     0   &     0    &    1   &     0    &    0  &      0   \\
     0   &     0    &    0   &     1    &    0  &      0   \\
     0   &     0    &    0   &     0    &    1  &      0
\end{array}
\begin{array}{rrrrrrr}
    0    &    10    &    10    &    11     &    0   &      0    \\
    0    &     7    &    16    &     0     &   15   &      0    \\
   16    &     2    &     0    &     2     &    3   &      0    \\
    1    &     0    &    23    &     3     &    0   &      9    \\
    0    &    12    &     2    &    11     &    0   &      2
\end{array}\right).
\]

$(n,N) = (5,13)$:
\[
A=\left(
\begin{array}{rrrrrrr}
    1   &     0    &    0    &    0   &     0   &     6   &     0  \\
    0   &     1    &    0    &    0   &     0   &     6   &     0  \\
    0   &     0    &    1    &    0   &     0   &     0   &     9  \\
    0   &     0    &    0    &    1   &     0   &     0   &    16  \\
    0   &     0    &    0    &    0   &     1   &     0   &     0
\end{array}
\begin{array}{rrrrrrr}
  4   &     12  &      16    &     0    &     0   &      0       \\
  8   &      5  &       0    &     0    &    15   &      0       \\
  5   &      0  &       0    &    11    &    12   &      0       \\
  0   &      6  &       1    &     0    &     0   &      8       \\
  7   &      6  &       0    &    10    &     0   &      9
\end{array}\right).
\]

$(n,N) = (5,14)$:
\[
A=\left(
\begin{array}{rrrrrrr}
     1     &  0   &     0  &      0  &      0  &    11     &   0  \\
     0     &  1   &     0  &      0  &      0  &     5     &   0  \\
     0     &  0   &     1  &      0  &      0  &     0     &   3  \\
     0     &  0   &     0  &      1  &      0  &     0     &  17  \\
     0     &  0   &     0  &      0  &      1  &     0     &   0
\end{array}
\begin{array}{rrrrrrr}
  20      &  0    &   16     &   4    &    0     &   0     &   0   \\
   0      &  1    &   16     &   0    &    0     &   4     &   0   \\
   6      &  0    &    0     &   0    &   13     &   8     &   0   \\
   0      &  0    &    8     &   8    &    0     &   0     &   4   \\
   0      &  1    &    2     &   0    &    1     &   0     &   3
\end{array}\right).
\]

$(n,N) = (5,15)$:
\[
A=\left(
\begin{array}{rrrrrrrr}
     1    &   0    &   0    &   0   &    0   &   12    &   0   &    4    \\
     0    &   1    &   0    &   0   &    0   &   17    &   0   &    0    \\
     0    &   0    &   1    &   0   &    0   &    0    &   1   &    8    \\
     0    &   0    &   0    &   1   &    0   &    0    &   3   &    0    \\
     0    &   0    &   0    &   0   &    1   &    0    &   0   &    0
\end{array}
\begin{array}{rrrrrrr}
    0     &    7    &     0     &   13     &    0    &     0    &     0     \\
    3     &    0    &    10     &    0     &    0    &     2    &     0     \\
    0     &    0    &     0     &    0     &   12    &    17    &     0     \\
    0     &    0    &     1     &   15     &    0    &     0    &     2     \\
    3     &    1    &     0     &    0     &   13    &     0    &    18
\end{array}\right).
\]

\section{Phase-retrievable subspaces}

We first prove the following special case.

\begin{prop} \label{prop-1.1}Let $\mathcal{F} =\{f_{i}\}_{i=1}^{n}$ be a basis for $H$. Then there exists a $k$-dimensional  maximal $\mathcal{F}$-PR subspace if and only if $1\leq k \leq [(n+1)/2]$.
\end{prop}
\begin{proof} Suppose that $M$ is a $k$-dimensional $\mathcal{F}$-PR subspace. Then we have that $n \geq 2k -1$ and hence $k \leq (n+1)/2$. For the other direction, note that for each invertible operator $T$ on $H$,  $M$ is  an maximal $\mathcal{F}$-PR subspace if and only if $(T^{t})^{-1}M$ is an maximal $T\mathcal{F}$-PR subspace. So it suffices to show that for each $k$-dimensional subspace $M$ with $1\leq k \leq [(n+1)/2]$ there exists a basis $\{u_{i}\}_{i=1}^{n}$ such that $M$ is an maximal PR subspace with respect to $\{u_{i}\}_{i=1}^{n}$.

Let  $\{\varphi_{j}\}_{j=1}^{2k-1} \subset M$ be a PR-frame for $M$. Without losing the generality we can assume that $\{\varphi_{1}, ... , \varphi_{k}\}$ is an orthonormal basis for $M$. Extend it to an orthonormal basis $\{e_{i}\}_{i=1}^{n}$ for $H$, where $e_{i} = \varphi_{i}$ for $i = 1, ... , k$. Define $u_{i}$ by
$$
u_{i} = e_{i} \ \ (i=1, .. , k, 2k, ... , n)  \ \ \text{and} \  u_{i} = e_{i} + \varphi_{i} \ \ (i =k+1, ... , 2k-1).
$$
Let $P_{M}$ be the orthogonal projection onto $M$. Clearly we have $$\{P_{M}u_{i}\}_{i=1}^{n} = \{\varphi_{1}, ... , \varphi_{2k-1}, 0, ... , 0\}, $$ and hence $\{u_{i}\}_{i=1}^{n}$ is a phase-retrievable for $M$. It is also easy to verify that $\{u_{i}\}_{i=1}^{n}$ is a basis for $H$.  Now we show that $M$ is an maximal PR  subspace with respect to $\{u_{1}, ..., u_{n}\}$. Let $\tilde{M} = \mathrm{span}\, \{M, u\}$ with $u = \sum_{j=k+1}^{n}a_{j}e_{j}$ in $ M^{\perp}$ and $||u|| = 1$.  Then $P_{\tilde{M}}u_{i}  = e_{i}$ for $1\leq i \leq k$, $ P_{\tilde{M}}u_{i}  = \varphi_{i} + a_{i}u$ for $k+1\leq i\leq 2k-1$ and $P_{\tilde{M}}u_{i}  = a_{i}u$ for $i \geq 2k-1$. If $a_{i} = 0$ for $i = 2k , ... , n$, then  $\{P_{\tilde{M}}u_{i}\}_{i=1}^{n}$ is not phase-retrievable for $\tilde{M}$ since it only contains at most $2k-1$ nonzero elements.  If $a_{i_{0}} \neq 0$ for some $i_{0} \geq 2k$, then clearly $\{P_{\tilde{M}}u_{i}\}_{i=1}^{n}$ is phase-retrievable for $\tilde{M}$ if and only if $\{P_{\tilde{M}}u_{i}\}_{i=1}^{2k-1}\cup \{a_{i_{0}}u\}$ is phase-retrievable for $\tilde{M}$. Thus  $\tilde{M}$ is not a PR subspace with respect to $\{u_{1}, ..., u_{n}\}$ since we need at least $2k+1$ number of elements in a phase-retrievable frame for the $(k+1)$-dimensional space $\tilde{M}$.
\end{proof}


Now lets consider the general frame case: Let $\mathcal{F}$ be a frame for $H$.  For each subset $\Lambda$ of $\{1,... , N\}$, let $$d_{\Lambda} = \max \{\dim \mathrm{span}\,  (\mathcal{F}_{\Lambda}), \dim \mathrm{span}\,  (\mathcal{F}_{\Lambda^{c}})\}. $$Define
$$
d(\mathcal{F}) = \min \{d_{\Lambda}: \Lambda \subset \{1, ... , N\}\}.
$$

\begin{theo} \label{main-thm-2} Let $\mathcal{F}$ be a frame for $H$. Then $k$ is the largest integer such that there exists a $k$-dimensional  maximal $\mathcal{F}$-PR subspace if and only if $k = d(\mathcal{F})$.
\end{theo}

Clearly, $d(\mathcal{F}) = n$ if and only if $\mathcal{F}$ has the complement property. Thus the above theorem is a natural generalization of Proposition \ref{thm-1.1} . We need to following lemma for the proof of Theorem \ref{main-thm-2}.

\begin{lem} \label{lem2.1} Let $Tx = \sum_{i=1}^{k}\langle x, x_{k}\rangle x_{k}$ be a rank-$k$ operator and $M$ be a  subspace of $H$ such that $\dim TM = k$, then $\dim P(M) = k$, where $P$ is the orthogonal projection onto $\mathrm{span}\, \{x_{1}, ... , x_{k}\}$.
\end{lem}
\begin{proof} Since $\langle x, x_{k}\rangle  = \langle Px, x_{k}\rangle$, we get that $range(T|_{M}) = range(T|_{PM})$. Thus $\dim P(M) \geq k$ and hence $\dim P(M) = k$.
\end{proof}
\vspace{3mm}

\noindent{\bfseries Proof of Theorem \ref{main-thm-2}.}  Clearly we only need to prove that if $d(\mathcal{F}) = k$, then there exists a $k$-dimensional $\mathcal{F}$-PR subspace and every $(k+1)$-dimensional subspace is not phase-retrievable with respect to $\mathcal{F}$.

Suppose that $M$ is a $(k+1)$-dimensional subspace of $H$ and it is also phase-retrievable with respect to $\mathcal{F}$. Then, by  Proposition \ref{thm-1.1}, we get that $d(P\mathcal{F})  = k+1$, and hence $d(\mathcal{F}) \geq d(P\mathcal{F}) \geq k+1$, which leads to a contradiction. Therefore every $(k+1)$-dimensional subspace is not phase-retrievable with respect to $\mathcal{F}$.

Next we show that there exists a $k$-dimensional $\mathcal{F}$-PR subspace. Let $\Omega$ be a subset of $\{1, ... , N\}$ be such that $\dim H_{\Omega} \geq k$, where $H_{\Omega} =\mathrm{span}\,  {\mathcal F}_{\Omega}$.  For $X = (x_1, ..., x_{k})\in H^{(k)}:=H\oplus ...  \oplus H$, define $T_{X}(z) = \sum_{i=1}^{k}\langle z, x_{k}\rangle x_{k}$.

Consider the following set
$$
S_{\Omega} = \{(x_{1}, ... , x_{k})\in H^{(k)}: \dim T_{X}(H_{\Omega}) = k\}.
$$

Since $\dim \mathrm{span}\,  {\mathcal F}_{\Omega}\geq k$, we get that there exists a linearly independent set $(f_{i_{1}}, ... , f_{i_{k}})$ in $\mathcal{F}_{\Omega}$. This implies that  $(f_{i_{1}}, ... , f_{i_{k}}) \in S_{\Omega}$  and hence $S_{\Omega}$ is not empty.

Moreover,  since $ \dim T_{X}(H_{\Omega}) = k$ if and only if there exists an $k\times k$ submatrix of  the $n\times |\Omega|$ matrix $[T_{X}f_{\omega}]$ whose determinant is a nonzero polynomial of the input variables $x_{1}, ... , x_{k}$, we obtain that  $S_{\Omega}$ is open dense in $H^{(k)}$.

Now for each subset $\Lambda $ in $\{1, ... , N\}$. Let $\Omega_{\Lambda} = \Lambda$ if $d_{\Lambda} = \dim \mathrm{span}\,  (\mathcal{F}_{\Lambda})$, and otherwise $\Omega_{\Lambda} = \Lambda^{c}$. Thus we have $\dim \mathrm{span}\, \mathcal{F}_{\Omega_{\Lambda}} \geq k$ for every subset $\Lambda$. Since each $S_{\Omega_{\Lambda}}$ is open dense in $H^{(k)}$, we get that $$S :=\bigcap_{\Lambda\subset \{1, ... , N\}}S_{\Omega_{\Lambda}}$$ is open dense in $H^{(k)}$. Let $X=(x_1, ..., x_{k}) \in S$ and $M = \mathrm{span}\,  \{x, ... , x_{k}\}$. Then by Lemma \ref{lem2.1} we obtain that $\dim P(H_{\Omega_{\Lambda}}) = k$. This implies that either $\dim \mathrm{span}\,  P\mathcal{F}_{\Lambda} = k$ or $\dim \mathrm{span}\,  P\mathcal{F}_{\Lambda^{c}} = k$ for each subset $\Lambda$. Hence $\{Pf_{j}\}_{j=1}^{N}$ is a frame for $M$ that has the complement property, which implies  by Proposition \ref{thm-1.1} that  $M$ is a $k$-dimensional $\mathcal{F}$-PR subspace. \qed

From the proof of Theorem \ref{main-thm-2}, we also have the following:

\begin{coro}  Let $\mathcal{F}$ be a  frame for $H$. Then for almost all the  vectors $(x_{1}, ... x_{\ell})$ in $H^{(\ell)}$ (here $\ell \leq  d(\mathcal{F})$, the subspace $\mathrm{span}\, \{x_{1}, ... , x_{\ell}\}$ is  phase-retrievable with respect to $\mathcal{F}$. More precisely, for each $\ell \leq d(\mathcal{F})$, the following set $$\{ (x_{1}, ... x_{\ell})\in H^{(\ell)}: \ \mathrm{span}\,  \{x_{1}, ... , x_{\ell}\}\  \text{ is} \ \text{phase retrievable with respect to}\  \mathcal{F}\}$$
is open dense in $H^{(\ell)}$.
\end{coro}

The following lemma follows immediately from the definitions, and it tells us that it is enough to focus on maximal phase-retrievable subspaces for frames with the exact PR-redundancy property.

\begin{lem} \label{lem2.2} Let $\mathcal{F} = \{f_{i}\}_{i=1}^{N}$ be a frame for $H$, and $\Lambda \subset \{1, ... , N\}$. If $ker(\Theta_{L(\mathcal{F}_{\Lambda})})\cap \mathcal{S}_{2}  = ker(\Theta_{L(\mathcal F)})\cap \mathcal{S}_{2} ,$  then $M$ is a $\mathcal{F}$-PR subspace if and only if  it is a $\mathcal{F}_{\Lambda}$-PR subspace. Consequently, $M$ is an maximal  $\mathcal{F}$-PR subspace if and only if it is an maximal $\mathcal{F}_{\Lambda}$-PR subspace.
\end{lem}


Now we would like to know what are the possible values of $d(\mathcal{F})$. Since every frame contains a basis, we get by Proposition \ref{prop-1.1} that  $k\geq [\frac{n+1}{2}]$. The following theorem tells us that for every $k$ between $[(n+1)/2]$ and $n$, there is a frame $\mathcal{F}$ with the exact PR-redundancy property such that  $k = d(\mathcal{F})$.

\begin{theo} \label{main-thm-3} Let $H = \mathbb{R}^{n}$ and $k$ be an integer such that $n\geq k \geq [\frac{n+1}{2}]$. Then  for each $N$ between $ 2k-1$ and $ k(k+1)/2+ (n-k)(n-k+1)/2$, there exists a frame $\mathcal{F}$ of length $N$ such that it has the exact PR-redundancy property and  $d(\mathcal{F}) =k$, i.e., $k$ is the largest integer such that there exists a $k$-dimensional  maximal $\mathcal{F}$-PR subspace.
\end{theo}

Before giving the proof we remark  while the proof of the this theorem uses Theorem \ref{main-thm-1}, it is also a generalization of Theorem \ref{main-thm-1} since it clearly recovers Theorem \ref{main-thm-1} if we let $n=k$.

\begin{proof} Let $M$ be a $k$-dimensional subspace of $H$. For each

We divide the proof into two cases.

Case (i).  Assume that $2k-1\leq N \leq k(k+1)/2$.

By Theorem \ref{main-thm-1}, there exists an exact PR-frame $\mathcal{G} = \{g_{i}\}_{i=1}^{N}$ for $M$. Without losing the generality we can also assume that $\{g_{1}, ... , g_{k}\}$ is an orthonormal basis for $M$. Extend it to an orthonormal basis $\{e_{i}\}_{i=1}^{n}$ with $e_{1} = g_{1} , ... , e_{k} = g_{k}$. Let  $$\mathcal{F} = \{f_{i}\}_{i=1}^{N} = \{e_{1}, ... , e_{k}, g_{k+1} + e_{k+1}, ... , g_{n}+ e_{n}, g_{n+1}. ... , g_{N}\}.$$ Then it is a frame for $H$. Consider the  subset $\Lambda = \{1, ..., k, n+1, ... , N\}$ of $\{1, ... , N\}$. We have $\dim \mathrm{span}\,  \mathcal{F}_{\Lambda} = \dim  M = k$, and $\dim \mathrm{span}\,  \mathcal{F}_{\Lambda^{c}} \leq n - k$.  Note that from $k \geq [\frac{n+1}{2}]$ we get that $n-k\leq k$. Thus we have
$d(\mathcal{F}) \leq \max \{n-k, k\} = k$. On the other hand, it is easy to prove that $d(\mathcal{F}) \geq d(P_{M}\mathcal{F}) = d(\mathcal{G}) = k$, where $P_{M}$ is the orthogonal projection onto $M$. Therefore we have $d(\mathcal{F})  = k$.

Now we show  that $\mathcal{F}$ has the exact PR-redundancy property. If fact, if $\Lambda$ is a proper subset of $\{1, ... , N\}$, then $P_{M}\mathcal{F}_{\Lambda}$ is not a PR frame for $M$ since $P_{M}\mathcal{F} = \mathcal{G}$ is an exact PR-frame for $M$.  Therefore, there exists $x$ and $y$ in $M$ such that $|\langle x, P_{M}f_{i} \rangle | =
|\langle y, P_{M}f_{i}\rangle | $ for all $i\in \Lambda$ and  $ A= x\otimes x - y\otimes y \neq 0$.  Since $P_{M}\mathcal{F}$ is a PR-frame for $M$, we obtain that
$|\langle x, P_{M}f_{i}\rangle | \ne  |\langle y, P_{M}f_{i}\rangle | $ for some $i\in\Lambda^{c}$.  Note that $|\langle z, f_{i}\rangle  = \langle z, P_{M}f_{i}\rangle $ for every $z\in M$. Therefore, we have that $A\in ker \Theta_{L(\mathcal{F}_{\Lambda})}\cap S_{2}$ but $A\notin ker \Theta_{L(\mathcal{F})}\cap S_{2}$, and hence $ ker \Theta_{L(\mathcal{F}_{\Lambda})}\cap S_{2} \neq  ker \Theta_{L(\mathcal{F})}\cap S_{2}$ for any proper subset $\Lambda$. So $\mathcal{F}$  has the exact PR-redundancy property.

Case (ii): Assume that $k(k+1)/2< N \leq  k(k+1)/2+ (n-k)(n-k+1)/2$.

Since $k \geq [(n+1)/2] \geq n/2$, it is easy to verify that $$k(k+1)/2 \geq (2k-1) + 2(n-k)-1 = 2n-2. $$ Then we can write $N = N_{1} + N_{2}$ such that $$2k-1\leq N_{1} \leq k(k+1)/2 \ \ \text{and }\ 2(n-k)-1 \leq N_{2} \leq (n-k)(n-k+1)/2.$$

By Theorem \ref{main-thm-1}, there exist an exact PR-frame $\mathcal{F}_{1}$ of length $N_{1}$ for $M$ and an exact PR-frame $\mathcal{F}_{2}$ of length $N_{2}$ for the $M^{\perp}$. By Lemma \ref{lem-02},
we know that $\mathcal{F} = \mathcal{F}_{1}\cup \mathcal{F}_{2}$ is a frame of length $N$ with the exact PR-redundancy property.  Clearly $d(\mathcal{F}) \leq k$ since  $$\max \{\dim \mathrm{span}\,  \mathcal{F}_{1} , \dim \mathrm{span}\,  \mathcal{F}_{2}\} = k.$$ On the other hand, since $\mathcal{F}$ has a $k$-dimensional PR-subspace $M$,  we get from Theorem \ref{main-thm-2} that $d(\mathcal{F}) \geq k$. Thus we have $d(\mathcal{F}) = k$.
\end{proof}

The following example shows that $k(k+1)/2+ (n-k)(n-k+1)/2$ is not necessarily the upper bound of $N$ such that  there exists a frame $\mathcal{F} $ of length $N$ with the exact PR-redundancy property and $d(\mathcal{F}) = k$.

\begin{exam} Let $\{e_{1}, e_{2}, e_{3}\}$ be an orthonormal basis for $\R^{3}$. Consider the following frame
$$
\mathcal{F} = \{e_{1}, e_{2}, e_{3},  e_{1} + e_{2},  e_{1}+e_{2} + e_{3}\}.
$$
Then $k = d(\mathcal{F}) = 2$ and $5 > k(k+1)/2+ (3-k)(3-k+1)/2 =4$ . We can check that $\mathcal{F}$ has the exact PR-redundancy property.  Let $\mathcal{G}$ be the frame after removing an element $f$ from $\mathcal{F}$. Based on the following five cases, we can easily construct $A = x\otimes x - y\otimes y$  such that $A\neq 0$, $A\in ker \Theta_{L(\mathcal{G})}$ but $A\notin ker \Theta_{L(\mathcal{F})}$:

(i) $f = e_{1}$: Let $x = 2e_{1} = e_{2}$ and $y = 4e_{1} - e_{2}$.

(ii) $f = e_{2}$: Let $x = 2e_{2} + e_{1}$ and $y = 4e_{2} - e_{1}$.

(iii) $f = e_{3}$: Let $x = e_{1} + e_{3}$ and $y = e_{1} - 3e_{3}$.

(iv) $f = e_{1} + e_{2}$: Let $x = e_{1} + (e_{2}+e_{3})$ and $y = e_{1} -(e_{2}+e_{3})$

(v) $f = e_{1}+e_{2} + e_{3}$: Let $x = e_{1}+ e_{3}$ and $y = e_{1} - e_{3}$.
\end{exam}

\begin{prop}Let $H = \mathbb{R}^{n}$. Suppose that a frame $\mathcal{F}$ of length $N$ has the exact PR-redundancy property and $d(\mathcal{F}) <n$. Then $N < n(n+1)/2$.
\end{prop}
\begin{proof} Since  $\mathcal{F}$ has the exact PR-redundancy, we get that $N \leq n(n+1)/2$. If $N = n(n+1)/2$, then, by Lemma \ref{lem-00}, $\{f_{i}\otimes f_{i}\}_{N}$ is linearly independent and hence a basis for  $\mathbb{H}_{n}$. This implies that $\mathcal{F}$ is phase-retrievable and so $d(\mathcal{F}) = n$. This contradiction shows that $N < n(n+1)/2$.
\end{proof}

\noindent{\bfseries Question.} Give an integer $k$ such that $n > k \geq [\frac{n+1}{2}]$. What is the least upper bound $N$ such that there exists a frame $\mathcal{F}$ of length $N$ which has  the exact PR-redundancy property and  $d(\mathcal{F}) =k$?



\section{Maximal Phase-Retrievable Subspaces with respect to bases}

Given a basis $\mathcal{F} =\{f_{1}, ... , f_{n}\}$.  We would like to have a better understanding about the maximal phase-retrievable subspaces with respect to $\mathcal{F}$. We will first focus on orthonormal bases  and then use the similarity to pass to general bases.

Now we assume that $\mathcal{E} =\{e_{1}, ... , e_{n}\}$ is an orthonormal basis for $\Bbb{R}^{n}$. By  Proposition \ref{prop-1.1}, we know that there exists a $k$-dimensional maximal  $\mathcal{E}$-PR subspace for ever integer $k$ with $1\leq k \leq [{n+1\over 2}]$. What more can we say about these  $k$-dimensional maximal  $\mathcal{E}$-PR subspaces? We  explore its connections with the support property of vectors in these subspaces.  Recall that for a vector $x =\sum_{i=1}^{n}\alpha_i e_i \in \mathbb{R}^n$, the  support of $x$ is defined by $\text{supp}_{\mathcal E}(x):= \{i \ | \alpha_i \neq 0\}$. We will also use $\text{supp}(x)$ to denote $\text{supp}_{\mathcal E}(x)$ if $\mathcal{E}$ is well understood in the statements,  and use $|\Lambda|$ to denote the cardinality of any set $\Lambda$.

\begin{prop} \label{prop-4.1}
Suppose that  $ M$ is a $k$-dimensional $\mathcal{E}-PR $ subspace. Then for any nonzero vector $x \in M$,  we have $|\text{supp}(x)| \geq k$.
\end{prop}

\begin{proof} Assume to the contrary that  there exists a nonzero $x \in M$ with $|\text{supp}(x)|=j<k$. We may assume that $\|x\|=1$ and that $\text{supp}(x)=\{1,2,...,j\}$. Pick vectors $y_1,...,y_{k-1}$ in $M$ such that the set $\{x,y_1,...,y_{k-1}\}$ is an orthonormal basis for $M$. Then we have,

\begin{align*}
P_M(e_1) & = \langle e_1,x \rangle x + \langle e_1,y_1 \rangle y_1 + \cdots + \langle e_1,y_{k-1} \rangle y_{k-1}\\
P_M(e_2) & = \langle e_2,x \rangle x + \langle e_2,y_1 \rangle y_1 + \cdots + \langle e_2,y_{k-1} \rangle y_{k-1}\\
\vdots \\
P_M(e_j) &  = \langle e_j,x \rangle x + \langle e_j,y_1 \rangle y_1 + \cdots + \langle e_j,y_{k-1} \rangle y_{k-1}\\
P_M(e_{j+1}) & = \langle e_{j+1},y_1 \rangle y_1 + \cdots + \langle e_{j+1},y_{k-1} \rangle y_{k-1}\\
\vdots \\
P_M(e_n) & = \langle e_n,y_1 \rangle y_1 + \cdots + \langle e_n,y_{k-1} \rangle y_{k-1}.\\
\end{align*}
\noindent
The partition $\{P_M(e_1),...,P_M(e_j)\}$ and $\{P_M(e_{j+1}),...,P_M(e_n)\}$ does not have the complement property since the first set contains less than $k$ elements and the members of the second set are all contained in the $(k-1)$-dimensional subspace $\mathrm{span}\, \{y_1,...,y_{k-1}\}$. Thus $M$ is not  a $\mathcal{E}-PR$ subspace, which leads to a contradiction.
\end{proof}

\begin{coro} \label{corollary1}
If $M$ is a $k$-dimensional  $\mathcal{E}-PR$ subspace and there exists $x \in M$ such that $|\text{supp}(x)|=k$, then $M$ is maximal.
\end{coro}


Now suppose that $k \leq [(n+1)/2]$. Let $x\in H$ be a vector of norm one and  $|\text{supp}(x)| = k$. We show that $x$ can be extended to an orthonormal set $\{x, u_{1}, ... , u_{k-1}\}$ such that $M = \mathrm{span}\, \{x, u_{1}, ... , u_{k-1}\}$ is a $k$-dimensional $\mathcal{E}$-PR subspace.

\begin{theo} \label{Thm-ext} Let $u_{1}\in\R^{n}$ be a unit vector such that $|\text{supp}(u_{1})|=k$ and $k \leq [(n+1)/2]$. Then $u_{1}$ can be extended to an orthonormal set $\{u_{1}, ... , u_{k}\}$ such that $M = \mathrm{span}\, \{u_{1}, ... , u_{k}\}$ is a $k$-dimensional maximal $\mathcal{E}$-PR subspace.
\end{theo}

\begin{proof} We can assume that $\{e_{1}, ... , e_{n}\}$ is the standard orthonormal basis for $\R^{n}$ and $u_{1} = \sum_{i=1}^{k}\alpha_{i}u_{i}$ such that $\alpha_{i} \neq 0$ for every $1\leq i\leq k$.

It is easy to observe the following fact:    Let $m: 1\leq m\leq k$. Suppose that  $\{u_{1}, ... , u_{m}\}$ is an orthonormal set extension of $u_{1}$ and $$A(u_{1}, ... , u_{m}) = [u_{1}, ... , u_{m}].$$ Let $A_{\Lambda}(u_{1}, ... , u_{m})$ be the matrix consisting of row vectors corresponding to $\Lambda$.  If $A_{\Lambda}(u_{1}, ... , u_{m})$ is invertible  for every subset $\Lambda$ of $\{1, ... , n\}$ of cardinality $k$ with the property that $\Lambda\cap \{1, ... , k\} \neq \emptyset$, then the row vectors of $A(u_{1}, ... , u_{m})$ form a frame for $\R^{m}$ that has the complement property.

Now we use the induction to show that such an matrix $A(u_{1}, ... , u_{m})$ exists for every $m \in\{1, ... , k\}$. Clearly,  the $n\times 1$ matrix $A (u_{1})$  satisfies the requirement. Now assume that such an $n\times m$ matrix  $A(u_{1}, ... , u_{m})$ has been constructed and $m< k$. We want to prove that there exists a unit vector $u_{m+1}\perp u_{i} (1\leq i\leq  m)$ such that
$A(u_{1}, ... , u_{m}, u_{m+1})$ has the required property.

Let  $U = \text{span}\{u_{1}, ... , u_{m}\}^{\perp}$, and
let $\Lambda$ be a subset of $\{1, ... ,n\}$ such that $|\Lambda| = m+1$ and $\Lambda\cap \{1, ... , k\}\neq \emptyset$.  Define
$$\Omega_{\Lambda} = \{u\in U: A_{\Lambda}(u_{1}, ..., u_{m}, u) \ \text{is \ invertible}\}.$$
We claim that $\Omega_{\Lambda}$ is an open dense subset of $U$.

Using the fact that the set of invertible matrices form an open set in the space of all matrices, it is clear that $\Omega_{\Lambda}$ is open in $U$.

Now we show that $\Omega_{\Lambda} \neq \emptyset$. Let $\Lambda'$ be a subset of $\Lambda$ with cardinality $m$ and $\Lambda' \cap \{1, ... , k\} \neq \emptyset$. Then, by our induction assumption, we have that $A_{\Lambda'}(u_{1}, ... , u_{m})$ is invertible,  which implies that the $m$ column vectors of $A_{\Lambda}(u_{1}, ... , u_{m})$  form a  linearly independent set in the $m+1$ dimensional space $\R^{\Lambda} = \Pi_{i\in\Lambda}\R$.  Let $z\in \R^{m+1}$ be a nonzero vector such that it is orthogonal to all the column vectors of $A_{\Lambda}(u_{1}, ... , u_{m})$. Define
$u= (u_{1}, ... , u_{n})^{T}\in \R^{n}$ by letting $u_{i} = z_{i}$ for $i\in \Lambda$, and $0$ otherwise. Then $u\in U$ and hence $u\in \Omega_{\Lambda}$. Therefore we get that  $\Omega_{\Lambda}\neq\emptyset$.

For the density of $\Omega_{U}$, let $y\in U$ be an arbitrary vector and pick a vector $u\in \Omega_{\Lambda}$. Consider the vector $u_{t} = tu + (1-t)y\in U$ for $t\in  \R$.  Since
$A_{\Lambda}(u_{1}, ..., u_{m}, u)$ is invertible, we have that $det( A_{\Lambda}(u_{1}, ..., u_{m}, u_{t}))$ is a nonzero polynomial of $t$, and hence it is finitely many zeros. This implies that there exists a sequence $\{t_{j}\}$ such that $u_{t_{j}}\in \Omega_{\Lambda}$ and $\lim_{j\rightarrow \infty}t_{j} = 0$. Hence $u_{t_{j}} \rightarrow y$ and therefore $\Omega_{U}$ is dense in $U$.

By the Baire Category theorem we obtain that the intersection $\Omega$ of all such $\Omega_{\Lambda}$ is open dense in $Y$. Pick any $u_{m+1}\in\Omega$, then $A(u_{1}, ..., u_{m}, u_{m+1})$ has the required property. This completes the induction proof for the existence of  such an matrix $A  = [u_{1}, ... , u_{k}]$, where $\{u_{1}, ... , u_{k}\}$ is an orthonormal set extending the given vector $u_{1}$.

Write $u_{j} = (a_{1j}, a_{2j}, ... , a_{nj})^{T}$ for $1\leq j\leq k$. Let $M = {\text span}\{u_{1}, ... , u_{k}\}$ and $P$ be the orthogonal projection onto $M$. Then $$Pe_{i} = \sum_{j=}^{k}<e_{i}, u_{j}>u_{j} = \sum_{j=1}^{n}a_{ij}u_{j}$$ for all $1\leq i\leq n$. For every subset $\Lambda $ of $\{1, ... , n\}$, since  $\{u_{1}, ... , u_{k}\}$ is an orthonormal set, we have that $\{Pe_{j}: j\in \Lambda\}$  are linearly independent if and only if $A_{\Lambda}$ is invertible. Thus, $\{Pe_{i}\}_{i=1}^{n}$ has the complement property since the set of row vectors of $A$ has the complement property.
\end{proof}

\begin{rem} Note that from the proof it is easy to see  that the existence of such an matrix $A(u_{1}, ... , u_{k})$  does not require the condition $k\leq [(n+1)/2]$. However, the complement property of the row vectors for $\R^{k}$ does require this condition.
\end{rem}

We know that  if $M$ is a $k$-dimensional PR-subspace with respect to an orthonormal basis $\mathcal{E}$, then the condition $\min \{ |supp(x)| : 0\neq x \in M\} = k$  is sufficient for $M$ to be maximal. The following example show that this condition is not necessary in general. However, we will prove in Theorem \ref{thm-ext2} that it is indeed also necessary if $k < [{n+1\over 2}]$.

\begin{exam} There exists a $2$-dimensional maximal PR-subspace $M$ in $\R^{4}$ such that $|\text{supp}(x)| = 3$ for every nonzero $x\in M$. Indeed, let $\{e_{1},  e_{2}, e_{3}, e_{4}\}$ and orthonormal basis for $\R^{4}$ and  be  $M = \mathrm{span}\,  \{e_{1}+e_{2}+e_{3}, e_{1}- e_{2} + e_{4}\}$. Then it can be easily verified that $M$ is a PR-subspace and $ |\text{supp}(x)| = 3$ for every nonzero $x\in M$. It is clear that $M$ is maximal since there is no $3$-dimensional  PR-subspace with respect to $\{e_{1}, e_{2}, e_{3}, e_{4}\}$ in $\R^{4}$. It is interesting to note that if we view $M$ as a subspace in $R^{n}$ with $n\geq 5$, then $M$ is no longer maximal anymore.
\end{exam}

\begin{theo} \label{thm-ext2} Assume that $M = \mathrm{span}\, \{u_{1}, ... , u_{k}\}$ is a $k$-dimensional maximal PR-subspace with respect to $\{e_{1}, ...,  e_{n}\}$ and $k< [{n+1\over 2}]$. Then $k = \min \{ |supp(x)| : 0\neq x \in M\}$.
\end{theo}
\begin{proof} By Proposition \ref{prop-4.1}, it suffices to show there is a nonzero vector $x\in M$ such that $|supp(x)| \leq k$.

 Let $\{u_{1}, ... , u_{k}\}$ be an orthonormal basis for $M$. We adopt the notation used in the proof of Theorem \ref{Thm-ext}: For every subset $\Lambda$ of $\{1, ... , n\}$, let $A_{\Lambda}(u_{1}, ... , u_{k})$ be the matrix consisting of row vectors of $[u_{1}, ... , u_{k}]$ corresponding to the row index in $\Lambda$. It is obvious that if there is a subset $\Lambda$ with $|\Lambda| = n-k$ such that $\rank A_{\Lambda}(u_{1}, ... , u_{k}) < k$, then there is a nonzero vector $x\in M$ such that $supp(x) \subseteq \Lambda^{c}$ and hence  $|supp(x)| \leq k$. We will prove that such a subset $\Lambda$ exists.

 Assume, to the contrary, that $\rank A_{\Lambda}(u_{1}, ... , u_{k}) =  k$ for  any subset $\Lambda$ with $|\Lambda| = n-k$. Thus we have $\rank A_{\Lambda}(u_{1}, ... , u_{k}) =  k$ for  any subset $\Lambda$ with $|\Lambda| \geq n-k$.

 For each subset $\Lambda$, since $k< [{n+1\over 2}]$, we only have three possible cases:

 (i) $|\Lambda| \geq n-k$ and $|\Lambda^{c}| < n -k$.

 (ii) $|\Lambda^{c}| \geq n-k$ and $|\Lambda| < n -k$.

 (iii) $|\Lambda| < n-k$ and $|\Lambda^{c}| < n -k$.

 Note that  case (iii) implies that $|\Lambda| > k$ and $|\Lambda^{c}| > k$. Now we  assign each $\Lambda$ to a subset  $S(\Lambda)$ by the following rule:
 Set $S(\Lambda)$ to be $ \Lambda$ or $\Lambda^{c}$ depending case (i) or case (ii). Suppose that $\Lambda$ satisfies $(iii)$. Since the row vectors of $[u_{1}, ... , u_{k}]$ has the complement property, we have that either $\rank A_{\Lambda}(u_{1}, ... , u_{k}) =  k$ or $\rank A_{\Lambda^{c}}(u_{1}, ... , u_{k}) =  k$. In this case we set $S(\Lambda) = \Lambda$ if $\rank A_{\Lambda}(u_{1}, ... , u_{k}) =  k$, and otherwise set $S(\Lambda) = A^{c}$. Let $$\mathcal{S} = \big\{S(\Lambda): \Lambda \subseteq \{1, ... , n\}\big\}.$$
 Then for each $\Lambda$  we have either $S(\Lambda) = \Lambda$ or $S(\Lambda) = \Lambda^{c}$,   $\rank A_{S(\Lambda)}(u_{1}, ..., u_{k}) = k$ and $|S(\Lambda)|\geq k+1$.

Let  $U = \text{span}\{u_{1}, ... , u_{k}\}^{\perp}$ and $$\Omega_{\Lambda} = \{u\in U: \rank A_{S(\Lambda)}(u_{1}, ... u_{k}, u) = k+1\}.$$ Then by the exact same argument as in the proof of Theorem \ref{Thm-ext}, we get that $\Omega_{\Lambda}$ is open dense in $U$. The Baire-Category theorem  implies that there exists unit vector $u_{k+1}\in U$ such that
$\rank A_{S(\Lambda)}(u_{1}, ... u_{k}, u_{k+1}) = k+1$ for every subset $\Lambda\subseteq \{1, ... , n\}$. This shows that the row vectors of the matrix $[u_{1}, ... , u_{k}, u_{k+1}]$ has the complementary property, and hence $span\{u_{1}, .., u_{k}, u_{k+1}\}$ is a PR-subspace with respect to the orthonormal basis $\{e_{1}, ... , e_{n}\}$, which contradicts the maximality of $M$.
\end{proof}

\begin{exam}
Let $\mathcal{F}=\{e_1,...,e_n\}$ be an orthonormal basis for $\mathbb{R}^n$. Then $M = \mathrm{span}\, \{x\}$ be a one-dimensional  maximal $\mathcal{F}$-PR subspace if and only if $|\text{supp}(x)|=1$.
\end{exam}

\begin{exam} Let $x\in\R^{n}$ be a unit vector such that $|\text{supp}(x)|=2$ and  M be a $2$-dimensional subspace containing $x$. Then M is maximal $\mathcal{F}$-PR subspace if and only if there exists an orthonormal basis $\{x,y\}$ for M such that $y=y_1+y_2$ with
$0\neq y_1 \in \mathrm{span}\, \{e_i: i \in supp(x)\}$ and $0\neq y_2 \in \mathrm{span}\,  \{e_i: i \notin supp(x) \}$. Indeed, by Corollary \ref{corollary1}, it suffices to show that $M$ is a $\mathcal{F}-PR$ subspace. We may assume that $\text{supp}(x)=\{1,2\}$.  Then we have

\begin{align*}
P_M(e_1) & = \langle e_1,x \rangle x + \langle e_1, y_{1} \rangle y \\
P_M(e_2) & = \langle e_2,x \rangle x + \langle e_2, y_{1}  \rangle y \\
P_M(e_3) & = \langle e_3, y_{2} \rangle y \\
\vdots \\
P_M(e_n) & = \langle e_n, y_{2} \rangle y \\.
\end{align*}
Then it  is easy to check that $\{P_{M}e_{i}\}$ has the complement property if and only if $\{P_{M}e_{1}, P_{M}e_{2}\}$ are linearly independent, and $\langle e_i, y_{2} \rangle \neq 0$ for some $3\leq i\leq n$. This is in turn equivalent to the conditions that $y_{1}\neq 0$ and $y_{2} \neq 0$.
\end{exam}

Finally, let examine the general basis case. Let $\mathcal{F} = \{f_{1}, ..., f_{n}\}$ be a basis for $\R^{n}$, and $\mathcal{F^{*}} = \{f_{1}^{*}, ..., f_{n}^{*}\}$ be its dual basis. Let $T$ be the invertible matrix such that $f_{i} = Te_{i}$ for all $i$, where $\mathcal{E} = \{e_{1}, ... , e_{n}\}$ be the standard orthonormal basis for $\R^{n}$. We observe the following facts:

(i)  $M$ is a PR-subspace with respect to $\mathcal{F}$ if and only if $T^{t}M$ is a PR-subspace with respect to $\mathcal{E}$.

(ii) The dual basis $\mathcal{F}^{*} = \{(T^{-1})^{t}T^{-1}e_{i}\}_{i=1}^{n}$, i.e., $f_{i}^{*} = (T^{-1})^{t}T^{-1}e_{i}$.

(iii) The coordinate vector of $x$ with respect to the basis $\mathcal{F}^{*}$ is the same as the coordinate vector of $T^{t}x$ with respect to the basis $\mathcal{E}$.

Based on the above observations we summarize the main results of this section in the following theorem:

\begin{theo} \label{thm-egeneralcase} Let $\mathcal{F} = \{f_{1}, ..., f_{n}\}$ be a basis for $\R^{n}$, and $\mathcal{F}^{*}= \{f_{1}^{*}, ..., f_{n}^{*}\}$ be its dual basis. Then we have

(i) If $M$ is a $k$-dimensional PR-subspace with respect to $\mathcal{F}$, then $|\text{supp}_{\mathcal{F}^{*}}(x)| \geq k$ for any nonzero vector $x\in M$. Consequently,  $M$ is maximal if there exists a vector $x\in M$ such that $|\text{supp}_{\mathcal{F}^{*}}(x)| = k$.

(ii) For any vector $x\in\R^{n}$ such that  $|\text{supp}_{\mathcal{F}^{*}}(x)| = k$, there exists a $k$-dimensional maximal PR-subspace $M$ with respect to $\mathcal{F}$ such that $x\in M$.

(iii) If $k < [(n+1)/2]$ and $M$ is a $k$-dimensional PR-subspace with respect to $\mathcal{F}$, then $M$ is maximal if and only if  there exists a vector $x\in M$ such that
$|\text{supp}_{\mathcal{F}^{*}}(x)| = k$.
\end{theo}










\end{document}